\documentclass[11pt]{article}
\usepackage{amsmath,amsthm,amsfonts,array,comment}
\renewcommand{\=}{\doteq}

\newcommand{\T}{\mathfrak{T}}

\newcommand{\g}{\overline{g}}
\newcommand{\h}{\overline{h}}
\newcommand{\e}{\overline{e}}
\renewcommand{\gg}{g^{-1}}
\newcommand{\hh}{h^{-1}}
\newtheorem{thm}{Theorem}[section]
 \newtheorem{cor}[thm]{Corollary}
 \newtheorem{prop}[thm]{Proposition}
 \newtheorem{lemma}[thm]{Lemma}
\theoremstyle{definition}
 \newtheorem{defn}[thm]{Definition}
\theoremstyle{definition}

\theoremstyle{definition}
 
\numberwithin{equation}{section}
\numberwithin{equation}{section}
\begin{document}
\title{\bf  Moufang symmetry VIII.\\
Reconstruction of Moufang loops}
\author{Eugen Paal}
\date{}
\maketitle
\thispagestyle{empty}
\begin{abstract}
The reconstruction theorem for the Moufang loops is proved.
\end{abstract}

\section{Introduction}

In the present paper we prove the reconstuction theorem for the Moufang loops. 
We in part follow ideas presented in \cite{Doro} and \cite{Paal7}. 

\section{Moufang loops}

A \emph{Moufang loop} \cite{RM} (see also \cite{Bruck,Bel,HP}) is a set $G$ with a binary operation  (multiplication) 
$\cdot: G\times G\to G$, denoted also by juxtaposition, so that the following three axioms are satisfied:
\begin{enumerate}
\itemsep-3pt
\item[1)] 
in equation $gh=k$, the knowledge of any two of $g,h,k\in G$ specifies the third one \emph{uniquely},
\item[2)] 
there is a distinguished element $e\in G$ with the property $eg=ge=g$ for all $g\in G$,
\item[3)] 
the \emph{Moufang identity} 
\begin{equation*}
\label{Moufang}
(gh)(kg) = g(hk)g
\end{equation*}
holds in $G$.
\end{enumerate}
Recall that a set with a binary operation is called a \emph{groupoid}. A groupoid $G$ with axiom 1) is called a \emph{quasigroup}. If axioms 1) and 2) are satisfied, the grupoid (quasigroup) $G$ is called a \emph{loop}. The element $e$ in axiom 2) is called the \emph{unit} (element) of the (Moufang) loop $G$.

\section{Reconstruction Theorem}

\begin{thm}[reconstruction]
\label{rec-thm}
Let $G$ be a groupoid, $\T$ be a group with the unit element $E\in\T$, and $(S,TP)$ be a triple of maps $S,T,P:G\to \T$ such that:
\begin{itemize}
\itemsep-3pt
\item[1)]
for all $g$ in $G$ we have
\begin{equation}
\label{rec-thm-1}
S_gT_gP_g=E
\end{equation}
\item[2)]
for all $g$ in $G$ there exists $\g$ in $G$ such that 
\begin{equation}
\label{rec-thm-2}
S_{\g}\overset{(a)}{=}S^{-1}_g,\quad
T_{\g}\overset{(b)}{=}T^{-1}_g
\end{equation}
\item[3)]
for all $g,h$ in $G$ relations
\begin{align}
\label{rec-thm-3}
S_{\g h}\overset{(a)}{=}P_gS_hT_g,\quad
T_{\g h}\overset{(b)}{=}S_gT_hP_g,\quad
P_{\g h}\overset{(c)}{=}T_gP_hS_g\\
\label{rec-thm-4}
S_{h\g}\overset{(a)}{=}T_gS_hP_g,\quad
T_{h\g}\overset{(b)}{=}P_gT_hS_g,\quad
P_{h\g}\overset{(c)}{=}S_gP_hT_g
\end{align}
are satisafied in $\T$,
\item[4)]
from $S_g=S_h$ and $T_g=T_h$  it follows that $g=h$.
\end{itemize}
Then $G$ is a Moufang loop. The unit element of $G$ is $e\=g\g=\g g$, which does not depend on the choice of $g$ in $G$, and the inverse element of $g$ is $g^{-1}=\g$.
\end{thm}

We prove this theorem step by step. In what follows, $G$ denotes a groupoid. 

\section{Construction of unit end inverse elements}

\begin{prop}
We have $\overline{\g}=g$ for all $g$ in $G$.
\end{prop}

\begin{proof}
First calculate
\begin{equation*}
S_{\overline{\g}}
\overset{(\ref{rec-thm-2}a)}{=}S^{-1}_{\g}
=\left(S_{\g}\right)^{-1}
\overset{(\ref{rec-thm-2}a)}{=}\left(S^{-1}_{g}\right)^{-1}
=S_{g}
\end{equation*}
In the same way $T_{\overline{\g}}=T_{g}$. Finally use assumtion 4) of Theorem \ref{rec-thm} to get the desired relation.
\end{proof}

\begin{prop}
For all $g$ in $G$ we have
\begin{equation}
\label{STP-pre-inverse1}
S_{\g g}=T_{\g g}=P_{\g g}=E
\end{equation}
\end{prop}

\begin{proof}
In  (\ref{rec-thm-1}a--c) and (\ref{rec-thm-2}a--c) take $h=g$ and use assumption 1) of Theorem \ref{rec-thm}.
\end{proof}

\begin{prop}
We have
\begin{equation}
\label{P-inverse1}
P_{\g}=P^{-1}_g,\quad\forall g\in G
\end{equation}
\end{prop}

\begin{proof}
First calculate 
\begin{equation*}
E\overset{(\ref{STP-pre-inverse1})}{=}P_{\g g}
\overset{(\ref{rec-thm-4}c)}{=}S_{\g}P_{\g}T_{\g}
\overset{(\ref{rec-thm-2})}{=}S^{-1}P_{\g}T^{-1}_{g}
\end{equation*}
from which it follows that
\begin{equation}
\label{P-inverse2}
P_{\g}=S_gT_g
\end{equation}
Now calculate
\begin{gather*}
P_gP_{\g}=(T^{-1}_gS^{-1}_g)(S_gT_g)=E\\
P_{\g}P_g=(S_gT_g)(T^{-1}_gS^{-1}_g)=E
\end{gather*}
which imply the desired relation.
\end{proof}

\begin{prop}
We have
\begin{equation}
\label{ST-comm}
S_gT_g=T_gS_g,\quad \forall g\in G
\end{equation}
\end{prop}

\begin{proof}
Calculate
\begin{equation*}
S_gT_g
\overset{(\ref{P-inverse1})}{=}P_{\g}
\overset{(\ref{P-inverse2})}{=}T^{-1}_{\g}S^{-1}_{\g}
\overset{(\ref{rec-thm-2})}{=}T_gS_g
\tag*{\qed}
\end{equation*}
\renewcommand{\qed}{}
\end{proof}

\begin{prop}
We have
\begin{equation}
\label{TP-PS-comm}
T_gP_g\overset{(a)}{=}P_gT_g,\quad 
P_gS_g\overset{(b)}{=}S_gP_g,\quad \forall g\in G
\end{equation}
\end{prop}

\begin{proof}
Calculate
\begin{align*}
T_gP_g&=T_gT^{-1}_gS^{-1}_g=T^{-1}_gS^{-1}_gT_g=P_gT_g\\
P_gS_g&=P_gP^{-1}_gT^{-1}_g=P^{-1}_gT^{-1}_gP_g=S_gP_g
\tag*{\qed}
\end{align*}
\renewcommand{\qed}{}
\end{proof}

\begin{prop}
We have
\begin{equation}
\label{STP-pre-inverse2}
S_{g\g}=T_{g\g}=P_{g\g}=E,\quad \forall g\in G
\end{equation}
\end{prop}

\begin{proof}
Calculate
\begin{align*}
S_{g\g}
&\overset{(\ref{rec-thm-4}a)}{=}T_gS_gP_g
\overset{(\ref{ST-comm})}{=}S_gT_gP_g
\overset{(\ref{rec-thm-1})}{=}E\\
T_{g\g}
&\overset{(\ref{rec-thm-4}b)}{=}P_gT_gS_g
\overset{(\ref{TP-PS-comm}a)}{=}S_gT_gP_g
\overset{(\ref{rec-thm-1})}{=}E\\
P_{g\g}
&\overset{(\ref{rec-thm-4}c)}{=}S_gP_gT_g
\overset{(\ref{TP-PS-comm}b)}{=}S_gT_gP_g
\overset{(\ref{rec-thm-1})}{=}E
\tag*{\qed}
\end{align*}
\renewcommand{\qed}{}
\end{proof}

\begin{prop}
We have
\begin{equation}
\label{pre-inverse}
g\g=\g g,\quad \forall g\in g
\end{equation}
\end{prop}

\begin{proof}
Use
\begin{equation*}
S_{\g g}\overset{(\ref{STP-pre-inverse1})}{=}E\overset{(\ref{STP-pre-inverse2})}{=}S_{g\g},\quad 
T_{\g g}\overset{(\ref{STP-pre-inverse1})}{=}E\overset{(\ref{STP-pre-inverse2})}{=}T_{g\g}
\end{equation*}
with (\ref{rec-thm-2}).
\end{proof}

\begin{prop}
Element $g\g=\g g$ of $G$ does not depend on the choice of $g$ in $G$.
\end{prop}

\begin{proof}
It is sufficient to show that $g\g=h\h$ for all $g,h$ in $G$. The latter easily follows from
\begin{equation*}
S_{g\g}=E=S_{h\h},\quad S_{g\g}=E=S_{h\h}
\tag*{\qed}
\end{equation*}
\renewcommand{\qed}{}
\end{proof}

\begin{defn}
The uniquely defined element $g\g=\g g$ in $g$  is denoted as $e\=g\g=\g g$ .
\end{defn}

\begin{cor}
We have
\begin{equation}
S_e=T_e=P_e=E
\end{equation}
\end{cor}

\begin{prop}
The element $e$ in $G$ has the property that $\e=e$.
\end{prop}

\begin{proof}
Note that
\begin{align*}
E&=S^{-1}_e=S_{\e}=S_e\\
E&=T^{-1}_e=T_{\e}=T_e
\end{align*}
and use assumption 4) of Theorem \ref{rec-thm}.
\end{proof}

\begin{thm}
We have
\begin{equation}
eg=ge=g,\quad \forall g\in G
\end{equation}
\end{thm}

\begin{proof}
First use (\ref{rec-thm-3}) and (\ref{rec-thm-4}) to see that
\begin{align*}
S_{eg}&=P_eS_gT_e=S_g\\
T_{eg}&=S_eT_gT_e=T_g\\
S_{ge}&=T_eS_gP_e=S_g\\
T_{ge}&=P_eT_gS_e=T_g\\
\end{align*}
and use assumption 4) of Theorem \ref{rec-thm}.
\end{proof}

\begin{defn}[unit and inverse element]
We call $e$ the \emph{unit element} of $G$ and $\gg\=\g$ the \emph{inverse element} of $g$ in $G$.
\end{defn}

\section{Properties of inverse elements}

\begin{lemma}
\label{gx=h}
For given $g,h$ in $G$, element $\gg h$ of $G$ is a solution of equation $gx=h$, i.e 
\begin{equation*}
g(\gg h)=h
\end{equation*}
\end{lemma}

\begin{proof}
It is sufficient to check that
\begin{equation*}
S_{g(\gg h)}=S_h,\quad
T_{g(\gg h)}=T_h
\end{equation*}
Use (\ref{rec-thm-3}) to calculate
\begin{align*}
S_{g(\gg h)}
&=P_{\gg}S_{\gg h}T_{\gg}
=P_{\gg}P_gS_hT_gT_{\gg}
=S_h\\
T_{g(\gg h)}
&=S_{\gg}T_{\gg h}P_{\gg}
=S_{\gg}S_gT_hP_gP_{\gg}
=T_h
\tag*{\qed}
\end{align*}
\renewcommand{\qed}{}
\end{proof}

\begin{thm}
For given $g,h$ in $G$ equation $gx=h$ has the unique solution that coincides with $x=\gg h$. 
\end{thm}

\begin{proof}
We already know from lemma \ref{gx=h} that $x=\gg h$ is a solution of equation.$gx=h$. Let $y$ be another solution, i.e $gy=h$. We show that $y=\gg h$. We can see that
\begin{align*}
S_h
&=S_{gy}
=P^{-1}S_yT^{-1}_g\\
T_h
&=T_{gy}
=S^{-1}T_yP^{-1}_g
\end{align*}
which imply
\begin{align*}
S_y&=P_gS_hT_g=S_{\gg h}\\
T_y&=S_gT_hP_g=T_{\gg h}
\end{align*} 
It remains to use assumption 4) of Theorem \ref{rec-thm} finish the proof.
\end{proof}

By repeating the above proof we get

\begin{thm}
For given $g,h$ in $G$ equation $xg=h$ has the unique solution $x=h\gg$. 
\end{thm}

\begin{cor}[\cite{Bruck,Bel}]
Groupoid $G$ is a an inverse property loop (IP-loop).
\end{cor}

\begin{thm}
We have
\begin{equation*}
(gh)^{-1}=\hh\gg,\quad \forall g,h\in G
\end{equation*}
\end{thm}

\begin{proof}
It is sufficient to check that
\begin{equation*}
S_{(gh)^{-1}}=S_{\hh\gg},\quad
T_{(gh)^{-1}}=T_{\hh\gg}
\end{equation*}
Calculate
\begin{equation*}
S_{(gh)^{-1}}
=S^{-1}_{gh}
=\left(P_{\gg}S_hT_{\gg}\right)^{-1}
=T_gS^{-1}_hP_g
=T_gS_{\hh}P_g
=S_{\hh\gg}
\end{equation*}
The second relation can be checked analogously.
\end{proof}

\section{Flexibility and triple closure}

\begin{thm}[flexibility]
We have
\begin{equation*}
gh\cdot k=g\cdot hk,\quad \forall g,h\in G
\end{equation*}
\end{thm}

\begin{proof}
It is sufficient to check that
\begin{equation*}
S_{gh\cdot k}=S_{g\cdot hk},\quad
T_{gh\cdot k}=T_{g\cdot hk}
\end{equation*}
Calculate
\begin{align*}
S_{gh\cdot k}
&=T_{\gg}S_{gh}P_{\hh}
=T_{\gg}P_{\gg}S_{h}T_{\gg}P_{\hh}
=S_gS_hS_g\\
T_{gh\cdot k}
&=P_{\gg}T_{gh}S_{\hh}
=P_{\gg}S_{\gg}T_{h}P_{\gg}S_{\hh}
=T_gT_hT_g\\
S_{g\cdot hk}
&=P_{\gg}S_{hg}T_{\hh}
=P_{\gg}T_{\gg}S_{h}P_{\gg}T_{\hh}
=S_gS_hS_g\\
T_{g\cdot hk}
&=S_{\gg}T_{hg}P_{\hh}
=S_{\gg}P_{\gg}T_{h}S_{\gg}P_{\hh}
=T_gT_hT_g
\tag*{\qed}
\end{align*}
\renewcommand{\qed}{}
\end{proof}

From proof of this Theorem follows
\begin{thm}[triple closure]
We have the triple closure relations:
\begin{equation}
S_gS_hS_g=S_{ghk},\quad
T_gT_hT_g=T_{ghk},\quad
P_gP_hP_g=P_{ghk},\quad \forall g,h\in G
\end{equation}
\end{thm}

\section{Moufang identity}

\begin{thm}[Moufang identity]
In $G$ the Moufang identity is satisfied:
\begin{equation*}
(gh)(kg)=g(hk)g,\quad\forall g,h,k\in G
\end{equation*}
\end{thm}

\begin{proof}
It is sufficient to check that
\begin{equation*}
S_{(gh)(kg)}=S_{g(hk)g},\quad 
T_{(gh)(kg)}=T_{g(hk)g}
\end{equation*}
Calculate
\begin{align*}
S_{(gh)(kg)}
&=P_{(gh)^{-1}}S_{kg}T_{(gh)^{-1}}\\
&=P^{-1}_{gh}S_{kg}T^{-1}_{gh}\\
&=\left(T_{\gg}P_yS_{\gg}\right)^{-1}S_{kg}\left(S_{\gg}T_yP_{\gg}\right)^{-1}\\
&=S_gP^{-1}_hT_gS_{kg}P_gT^{-1}_{h}S_g\\
&=S_gP^{-1}_{h}S_{kg\cdot \gg}T^{-1}_{h}S_g\\
&=S_gP^{-1}_{h}S_{k}T^{-1}_{h}S_g\\
&=S_gS_{hk}S_g\\
&=S_{g(hk)g}
\end{align*}
Analogously,
\begin{align*}
T_{(gh)(kg)}
&=S_{(gh)^{-1}}T_{kg}P_{(gh)^{-1}}\\
&=S^{-1}_{gh}T_{kg}P^{-1}_{gh}\\
&=\left(P_{\gg}S_yT_{\gg}\right)^{-1}T_{kg}\left(T_{\gg}P_yS_{\gg}\right)^{-1}\\
&=T_gS^{-1}_hP_gS_{kg}S_gP^{-1}_{h}T_g\\
&=T_gS^{-1}_{h}T_{kg\cdot \gg}P^{-1}_{h}T_g\\
&=T_gS^{-1}_{h}T_{k}P^{-1}_{h}T_g\\
&=T_gS_{hk}T_g\\
&=T_{g(hk)g}
\tag*{\qed}
\end{align*}
\renewcommand{\qed}{}
\end{proof}

Theorem \ref{rec-thm} has been proved.

\section*{Acknowledgement}

Research was in part supported by the Estonian Science Foundation, Grant 6912.

\bigskip\noindent
Department of Mathematics\\
Tallinn University of Technology\\
Ehitajate tee 5, 19086 Tallinn, Estonia\\ 
E-mail: eugen.paal@ttu.ee


\begin{thebibliography}{9}
\itemsep-3pt

{\small

\bibitem{Bel}
V.~D.~Belousov.
\textit{Foundations of the Theory of Quasiqroups and Loops.}
Nauka, Moscow, 1967 (in Russian).

\bibitem{Bruck}
R. H. Bruck.
\textit{A Survey of Binary Systems.}
Springer, Berlin--Heidelberg--New York, 1971.

\bibitem{Doro}
S.~Doro.
Simple Moufang loops.
Math. Proc. Cambr. Phil. Soc. {\bf 83} (1978), 377-392

\bibitem{RM}
R.~Moufang.
Zur Struktur von Alternativk\"{o}rpern.
Math. Ann. {\bf B110} (1935), 416-430.

\bibitem{Paal7}
E. Paal.
Moufang symmetry VII. Moufang transformations.
Preprint http://arxiv.org/abs/0803.0242, 2008.

\bibitem{HP}
H.~Pflugfelder.
\textit{Quasigroups and Loops: Introduction.}
Heldermann Verlag, Berlin, 1990.

}

\end{thebibliography}
\end{document}